\PassOptionsToPackage{dvipsnames}{xcolor}
\documentclass{aic}

\usepackage{amsmath,amsthm,amssymb,mathrsfs,amstext,enumitem, hyperref}
\usepackage[colorlinks=true]{hyperref}
\newtheorem{thm}{Theorem}
\newtheorem{lem}{Lemma}
\newtheorem{defn}{Definition}
\newcommand{\N}{\mathbb{N}}

\aicAUTHORdetails{%
  title = {Monochromatic Translated Product and Answering Sahasrabudhe's Conjecture}, 
  author = {Sayan Goswami},
  plaintextauthor = {Sayan Goswami},
  plaintexttitle = {Monochromatic Translated Product and Answering Sahasrabudhe's Conjecture}, 
  runningtitle = {Monochromatic Translated Product}, 
  runningauthor = {Sayan Goswami},
  copyrightauthor = {Sayan Goswami},
  keywords = {Schur Theorem, $IP$ sets, Hindman Conjecture, Moreira's sum product theorem.},
}   

\aicEDITORdetails{%
   year={2026},
   volume={},
   number={2},
   received={21 December 2024},   
   revised={11 October 2025},    
   published={10 April 2026},  
   doi={10.19086/aic.2026.2},      
}   

\begin{document}

\begin{frontmatter}[classification=text]

\title{Monochromatic Translated Product and Answering Sahasrabudhe's Conjecture} 

\author[SG]{Sayan Goswami\thanks{Supported by  NBHM postdoctoral fellowship with reference no: 0204/27/(27)/2023/R \& D-II/11927.}}

\begin{abstract}
This article resolves two related problems in Ramsey theory on the integers. We show that for any finite coloring of the set of natural numbers, there exist numbers $a$ and $b$ for which the configuration $\{a, b, ab, a(b+1)\}$ is monochromatic. By redefining the variables $a=x$ and $ab=y,$ our configurations transforms into $\{x,y,x+y,\frac{y}{x}\}.$ This finding has two main consequences: first, it disproves a conjecture proposed by J. Sahasrabudhe; second, it establishes a quotient version of the long-standing Hindman's conjecture, which asks for a monochromatic set of the form $\{x,y,x+y,xy\}$. 
\end{abstract}
\end{frontmatter}

\section{Introduction}

Arithmetic Ramsey theory studies the existence of monochromatic configurations within any finite coloring of the integers or the natural numbers \( \mathbb{N} \). Here, a ``coloring'' refers to a finite disjoint partition of \( \mathbb{N} \), and a set is said to be \emph{monochromatic} if it is entirely contained in one part of this partition. A collection \( \mathcal{F} \) of subsets of \( \mathbb{N} \) is said to be \emph{partition regular}\footnote{Although this notion is often referred to as \emph{weakly partition regular} in the literature, we omit the term ``weakly'' throughout this paper.} if, for every finite coloring of \( \mathbb{N} \), there exists a monochromatic member \( F \in \mathcal{F} \).

In 1916, I.~Schur~\cite{schur} proved that the pattern \( \{a, b, a+b\} \) is partition regular. Applying the map \( n \mapsto 2^n \), one immediately obtains that \( \{a, b, a \cdot b\} \) is also partition regular. These two results are commonly referred to as the \emph{additive Schur theorem} and the \emph{multiplicative Schur theorem}, respectively.

For more than a century, it remained unknown whether these two phenomena could be unified in a single statement. That is, whether \( \{a, b, a+b, a \cdot b\} \) is also partition regular (see~\cite{erdos}).  In 1979, R.~L.~Graham and N.~Hindman independently established such a result for two-colorings (see~\cite[Section~4]{trans}). Subsequently, N.~Hindman conjectured that the statement should hold for all finite colorings, a problem now known as the \emph{Hindman conjecture}. The strongest known partial result toward this conjecture is due to J.~Moreira~\cite{m}, who proved that there exist \( a, b \in \mathbb{N} \) such that the set \( \{a, ab, a+b\} \) is monochromatic; see also~\cite{al1} for a concise proof.

In 2017, while investigating asymmetric variants of Schur’s theorem, J.~Sahasrabudhe~( see \cite[Question 31.]{saha}) conjectured that there exist finite colorings of the natural numbers that avoid monochromatic configurations of the form \( \{a, b, a(b+1)\} \). In this paper, we disprove this conjecture by showing that every finite coloring of \( \mathbb{N} \) necessarily contains such a configuration. In fact, we establish a stronger result: the larger set \( \{a, b, ab, a(b+1)\} \) can always be found monochromatic.

On the other hand, by redefining the variables $a=x, ab=y$ our result shows that the pattern $\{x,y,x+y,\frac{y}{x}\}$ is monochromatic.
Our result can therefore be viewed as a quotient version of the Hindman conjecture. The technique of our proof builds upon the argument of Moreira~\cite{m}, who employed van der Waerden’s theorem, whereas we make iterative use of Hindman’s theorem.

The following statement is our main result, which will be proved in the next section.

\begin{thm}\label{p}
For any finite coloring of \( \mathbb{N} \), there exist \( a, b \in \mathbb{N} \) such that the set \( \{a, b, ab, a(b+1)\} \) is monochromatic.
\end{thm}


\section{Proof of Theorem \ref{p}}

Before we proceed to the proof of our Theorem \ref{p}, we need to recall some technical definitions.

\begin{defn}\text{}
    \begin{enumerate}
        \item For any non-empty set $X,$ denote by $\mathcal{P}_f(X)$ the set of all non-empty finite subsets of $X.$
        \item For any sequence $\langle x_n\rangle_n$, denote by $$FS\left(\langle x_n\rangle_n\right)=\left\lbrace \sum_{n\in H}x_n:H\in \mathcal{P}_f(\N) \right\rbrace.$$

        \item A set $A\subseteq \N$ is called an IP set if there exists a sequence $\langle x_n\rangle_n$ such that $FS\left(\langle x_n\rangle_n\right)\subseteq A.$

         \item For any $R\in \N,$ a set $A\subseteq \N$ is called an $IP_R$ set if there exists a sequence $\langle x_n\rangle_{n=1}^R$ such that $FS\left(\langle x_n\rangle_{n=1}^R\right)\subseteq A.$

         \item For any $y\in \N,$ and $A\subseteq \N,$ denote by 

                       \begin{itemize}
                           \item $-y+A=\{n: n+y\in A\};$ and
                           \item $y^{-1}A=\{n: ny\in A\}.$
                       \end{itemize}
    \end{enumerate}
\end{defn}

To prove our result, we need the following three technical results.
The following lemma is a reformulation of the Hindman theorem \cite{finitesum}.
\begin{lem}\textup{\cite[Lemma 5.19.2.]{hs}}\label{1}
    Any finite coloring of an $IP$ set contains a monochromatic IP subset. 
\end{lem}

The proof of the next two lemmas is standard, but we give the proofs here for the sake of completeness.

\begin{lem}\label{2}
For any IP set \( A \subseteq \mathbb{N} \) and any natural number \( y \), there exists another IP set \( A' \subseteq A \) satisfying \( y^{-1}A' \subseteq \mathbb{N} \).
\end{lem}

\begin{proof}
By passing to a subsequence if necessary, we may assume that 
\(A\) contains \(FS(\langle x_n \rangle)\) for some sequence 
\(\langle x_n \rangle \subset (y\mathbb{N}+j)\), where 
\(j \in \{0,1,\ldots ,y-1\}\).

Now let \(F\) be any finite set with \(|F|=y\). Then
\[
x_F=\sum_{n\in F} x_n \in y\mathbb{N}.
\]
Thus every sum of exactly \(y\) distinct terms from the sequence 
is divisible by \(y\).

Selecting infinitely many distinct elements from the sequence 
\(\langle x_F\rangle\) and using them as generators, it follows that 
\(A\) contains another finite sums set
\[
A' = FS(\langle y_n\rangle),
\]
where each \(y_n\) is of the form
\[
y_n = y z_n
\]
for some integer \(z_n\). And hence $$FS(\langle z_n\rangle)= y^{-1}A'\subseteq \mathbb{N}.$$ This completes the proof.
\end{proof}

\begin{lem}\label{3}
For any \( IP \) set \( A \subseteq \mathbb{N} \), there exists \( y \in A \) such that 
\[
(-y + A) \cap A
\]
is also an \( IP \) set.
\end{lem}

\begin{proof}
Let \( A = FS(\langle x_n \rangle_n) \) be an \( IP \) set generated by a sequence \( \langle x_n \rangle_n \), i.e. 
\[
A = \left\{ \sum_{n \in F} x_n :  F \in \mathcal{P}_f(\N) \right\}.
\]
Take \( y = x_1 \in A \). Then any finite sum of distinct elements from \( (x_{n})_{n \ge 2} \) is of the form \( \sum_{n \in F} x_n \) with \( F \subseteq \{2,3,\dots\} \), and hence 
\[
\sum_{n \in F} x_n + y = \sum_{n \in F \cup \{1\}} x_n \in A.
\]
Thus, every element of \( FS(x_{n})_{n \ge 2} \) belongs to \( (-y + A) \cap A \). Since \( FS(x_{n})_{n \ge 2} \) is itself an \( IP \) set, the intersection \( (-y + A) \cap A \) contains an \( IP \) set, and hence is \( IP \).
\end{proof}

Now we are in the position to prove our Theorem \ref{p}.

\begin{proof}[Proof of  Theorem \ref{p}:] Suppose that \( \mathbb{N} \) is finitely colored, say  
\[
\mathbb{N} = \bigcup_{i=1}^r A_i.
\]
Without loss of generality, assume that \( A_1 \) is an \( IP \) set, and set \( D_0 = A_1 \).
By Lemma~\ref{3}, we can choose an element \( y_1 \in A_1 \) such that 
\( (-y_1 + A_1) \cap A_1 \) is an \( IP \) set. 
Then, by Lemma~\ref{2}, we can choose an $IP$ set
\[
C_1 \subset y_1^{-1}\big( (-y_1 + A_1) \cap A_1 \big),
\]
which is a subset of \( \mathbb{N} \). 
Finally, by Lemma~\ref{1}, there exists \( i_1 \in \{1, 2, \ldots, r\} \) such that
\[
D_1 = C_1 \cap A_{i_1}
     = y_1^{-1}\big( (-y_1 + A_1) \cap A_1 \big) \cap A_{i_1}
\]
is an \( IP \) set.

Using Lemmas~\ref{2} and~\ref{3}, choose \( y_2 \in D_1 \) such that 
\[
C_2 \subseteq y_2^{-1}\big( (-y_2 + D_1) \cap D_1 \big) \subset \mathbb{N}
\]
is an \( IP \) set. 
Then, by Lemma~\ref{1}, there exists \( i_2 \in \{1,2,\ldots,r\} \) such that 
\( D_2 = C_2 \cap A_{i_2} \) is an \( IP \) set.
Proceeding inductively, for each \( n \in \mathbb{N} \), choose \( y_n \in D_{n-1} \) and \( i_n \in \{1,2,\ldots,r\} \) such that the set 
    \[
    D_n = y_n^{-1}\big( (-y_n + D_{n-1}) \cap D_{n-1} \big) \cap A_{i_n}
    \]
    is an \( IP \) set.

By the Pigeonhole Principle, there exists \( k \in \{1,2,\ldots,r\} \) such that 
\( A_{i_j} = A_{i_n} = A_k \) for some indices \( j < n \).
Now choose \( x \in D_n \subseteq A_{i_n} = A_k \).

From the construction in~(2), we have
\[
x y_n \in D_{n-1} \subseteq y_{n-1}^{-1} D_{n-2} \subseteq \cdots \subseteq 
y_{n-1}^{-1} \cdots y_{j+1}^{-1} D_j.
\]
Hence,
\[
x y_n \cdots y_{j+1} \in D_j \subseteq A_{i_j} = A_k.
\]
Moreover, since \( y_n \in D_{n-1} \), applying the same argument gives
\[
y_n \cdots y_{j+1} \in D_j \subseteq A_{i_j} = A_k.
\]
Thus, both \( y_n \cdots y_{j+1} \) and \( x y_n \cdots y_{j+1} \) lie in the same color class \( A_k \).

Again, observe that 
\[
x y_n + y_n \in D_{n-1},
\] 
so by the same argument as above,
\[
(x y_n + y_n) \, y_{n-1} \cdots y_{j+1} \in D_j \subseteq A_{i_j}.
\]

Now define
\begin{enumerate}
    \item \( b = x \), and
    \item  \( a = y_n \cdots y_{j+1} \).
\end{enumerate}

Then \( a, b \in A_k \), and
\[
x y_n \cdots y_{j+1} = a b \in A_k, \quad \text{and} \quad 
(x y_n + y_n) \, y_{n-1} \cdots y_{j+1} = a(b+1) \in A_k.
\]

Hence, the set
\[
\{a, b, ab, a(b+1)\} \subseteq A_k
\]
is monochromatic. 

This completes the proof.

\end{proof}






\section*{Acknowledgments} 
The author is thankful to the referees for their comments on the previous draft of this paper and also for the proof of the Lemma \ref{2}. This paper is supported by NBHM postdoctoral fellowship with reference no: 0204/27/(27)/2023/R \& D-II/11927.

\bibliographystyle{amsplain}


\begin{aicauthors}
\begin{authorinfo}[SG]
  Sayan Goswami\\
  Ramakrishna Mission Vivekananda Educational and Research Institute (RKMVERI), \\
  Belur, India\\
  sayan92m\imageat{}gmail\imagedot{}com \\
\end{authorinfo}
\end{aicauthors}

\end{document}